\documentclass [11pt,] {article}
\usepackage{amsmath,amssymb}
\usepackage{amsthm,amsfonts,amscd,epsfig,lineno}
\usepackage{xcolor}
\usepackage{multirow}
\usepackage{cite,epic}
\usepackage{adjustbox}
\usepackage{graphicx}
\usepackage[left=2cm]{geometry}
\usepackage{lscape}

\usepackage{amsthm}

\usepackage{makecell}
\usepackage{float}
\usepackage{booktabs}
\usepackage{amsmath, amssymb, amsthm}

\newtheorem{theorem}{Theorem}[section]
\newtheorem{lemma}[theorem]{Lemma}

\newtheorem{corollary}[theorem]{Corollary}

\newtheorem{example}[theorem]{Example}

\usepackage{booktabs}
\usepackage{multirow}
\usepackage{makecell}
\usepackage{array}
\usepackage{float}

\usepackage{caption}
\newcount\refno
\usepackage{cite}
\numberwithin{equation}{section}

\title {Quantum Latin Hypercubes with Maximal Cardinality
}

\author {   \footnotesize     Mingzhen Lv,   Haitao Cao\footnote{Corresponding author. \; E-mail address: caohaitao@njnu.edu.cn.  }\\
       \scriptsize    School of Mathematical Sciences, Ministry of Education Key Laboratory for NSLSCS,\\ \scriptsize Nanjing Normal University, Nanjing 210023, China. }
\date{}

\title{Quantum Latin Hypercubes with Maximal Cardinality}

\begin {document}
\parindent=0.5cm
\baselineskip=0.6cm
\maketitle
\begin {abstract}
\baselineskip=0.6cm
We construct quantum Latin squares with maximal cardinality and quantum Latin hypercubes with maximal cardinality over the real field.

\noindent {\bf Keywords:} quantum Latin square, quantum Latin hypercube, maximal cardinality
\end{abstract}

\section{Introduction}

Quantum Latin squares are a natural quantum generalization of classical Latin squares, introduced by Musto and Vicary~\cite{Musto2016}. Quantum Latin squares and their higher-dimensional generalizations have found wide applications in quantum information, including unitary error bases~\cite{Musto2016}, mutually unbiased bases~\cite{Zang_2022}, $k$-uniform states~\cite{PhysRevA.97.062326,Zang_2021}, etc.

A quantum Latin square of order $q$ is a $q\times q$ array of quantum states, where each entry is a unit vector in the $q$-dimensional Hilbert space $\mathcal{H}_q$, such that every row and every column forms an orthonormal basis of $\mathcal{H}_q$. 

The cardinality of a quantum Latin square is defined as the number of its vectors that are distinct up to global phase, satisfying $q\le c\le q^2$.
Over the complex field, Zhang et al.~\cite{ZHANG2026114863,MR5034584} proved that for every $q\ge4$, except for a finite number of exceptions, there exists a quantum Latin square with maximum cardinality $q^2$. Subsequently, Zang et al.~\cite{MR5047492} resolved all remaining cases, establishing the existence for every $q\ge4$. For more on quantum Latin squares, see, among others, \cite{Mu1,Mu3,PhysRevA.104.042423,Rather2,Rather3,Zang1,Zang3,Zang5,Zhang2,Zhang4,Zhang5,Zyc}.

Goyeneche et al.~\cite{PhysRevA.97.062326} generalized quantum Latin squares to higher dimensions by introducing the notion of quantum Latin hypercubes: A quantum Latin hypercube of order $q$ and dimension $n$, denoted $\mathrm{QLH}(q,n)$, is an $n$-dimensional array of size $q\times q\times\cdots\times q$,
$
C=(|c_{i_0,i_1,\dots,i_{n-1}}\rangle),
$
where each entry is a unit vector in $\mathcal{H}_q$, such that every edge of the hypercube forms an orthonormal basis of $\mathcal{H}_q$.
Equivalently, for any axis direction $t\in\{0,1,\dots,n-1\}$ and any fixed coordinates $(i_0,\dots,i_{t-1},i_{t+1},\dots,i_{n-1})$, the set
$
\{|c_{i_0,\dots,i_{t-1},x,i_{t+1},\dots,i_{n-1}}\rangle:x\in\mathbb{Z}_q\}
$
forms an orthonormal basis of $\mathcal{H}_q$.
In particular, when $n=2$, a $\mathrm{QLH}(q,2)$ is a quantum Latin square of order $q$; when $n=3$, a $\mathrm{QLH}(q,3)$ is a quantum Latin cube of order $q$, denoted $\mathrm{QLC}(q)$.

The cardinality of a quantum Latin hypercube $\mathrm{QLH}(q,n)$ is defined as the number of its vectors that are distinct up to global phase. Clearly, the cardinality $c$ satisfies $q\le c\le q^n$. When $c=q$, the hypercube is called classical; when $c=q^n$, it has maximum cardinality.
Zhang et al.~\cite{MR5034584} have proved that there exists a
$\mathrm{QLC}(v)$
with maximal cardinality for all positive integers
$v \geq 4$, except when
$v$ is a prime
or
$v = pq$, where $p, q$
are prime numbers and
$(p, q) \neq (2, 2)$.
The remaining cases are still open.

In this paper, we first extend the aforementioned results to the real field, establishing the existence of maximum-cardinality quantum Latin squares over the real field.
Furthermore, we generalize this construction to the higher-dimensional setting, obtaining maximum cardinality quantum Latin hypercubes over the real field.
Our main results are as follows.

\begin{theorem}
\label{thm:1.1}
	For any positive integers $q\ge 4$, there exists a quantum Latin square of order $q$ over the real field with maximal cardinality $q^2$.
\end{theorem}

\begin{theorem}
\label{thm:1.2}
	For any positive integers $n$ and $q\ge n+2$, there exists a $\mathrm{QLH}(q,n)$ with maximal cardinality $q^n$.
\end{theorem}

\section{Quantum Latin Squares over the Real Field}

As mentioned in the introduction, the original construction of Zang  uses a fixed transposition of the second and third rows of the normalized Fourier matrix~\cite{MR5047492}. We now replace this fixed transposition with an arbitrary permutation $\sigma\in S_q$, thereby obtaining a family of constructions parameterized by permutations.

\begin{theorem}
\label{thm:2.1}
	For any positive integers $q$ and any permutation $\sigma\in S_q$, let $L=(|\ell_{i,j}\rangle)$ be a $q \times q$ array, where
	$$|\ell_{i,j}\rangle^{(k)}=\frac{1}{\sqrt{q}}\,\omega^{i\sigma(k)+jk},\qquad i,j,k\in\mathbb{Z}_q,
	\omega = e^{2\pi\mathrm{i}/q}.$$
	Then, $L$ is a $\mathrm{QLS}(q)$.
\end{theorem}

\begin{proof}
	We verify that the rows and columns of $L$ form orthonormal bases of $\mathbb{C}^q$.
	Fix $i\in\mathbb{Z}_q$. For any $j,j'\in\mathbb{Z}_q$, we have
	\begin{align*}
		\langle \ell_{i,j}|\ell_{i,j'}\rangle
		=\frac{1}{q}\sum_{k=0}^{q-1}
		\overline{\omega^{i\sigma(k)+jk}}\omega^{i\sigma(k)+j'k}
		=\frac{1}{q}\sum_{k=0}^{q-1}\omega^{(j'-j)k}
		=\delta_{jj'}.
	\end{align*}
	Thus each row forms an orthonormal basis.
	Fix $j\in\mathbb{Z}_q$. For any $i,i'\in\mathbb{Z}_q$, since $\sigma$ is a permutation of $\mathbb{Z}_q$, we have
	\begin{align*}
		\langle \ell_{i,j}|\ell_{i',j}\rangle
		=\frac{1}{q}\sum_{k=0}^{q-1}
		\overline{\omega^{i\sigma(k)+jk}}\omega^{i'\sigma(k)+jk}
		=\frac{1}{q}\sum_{k=0}^{q-1}\omega^{(i'-i)\sigma(k)}
		=\frac{1}{q}\sum_{k=0}^{q-1}\omega^{(i'-i)k}
		=\delta_{ii'}.
	\end{align*}
	Thus each column forms an orthonormal basis.
	Therefore, $L$ is a $\mathrm{QLS}(q)$.
\end{proof}

By choosing special permutations $\sigma$, one obtains $\mathrm{QLS}(q)$ with maximum cardinality $q^2$, as demonstrated in \cite{MR5047492}. In the following, we show that this method can also be adapted to construct maximum cardinality quantum Latin squares over the real field.

Before proving Theorem~\ref{thm:1.1}, we first give an example of a $\mathrm{QLS}(5)$ over the real field.

\begin{example}
\label{ex:q5}
	For $q=5$, taking $\sigma=(1\;\;2)(-1\;\;-2)$, by Theorem~\ref{thm:2.1}, we obtain a $\mathrm{QLS}(5)$ with maximum cardinality, denoted by $L=(|\ell_{i,j}\rangle)_{i,j\in\mathbb{Z}_5}$. Explicitly, the array is
	\[
	\begin{array}{c|c|c|c|c|c|}
		& j=0 & j=1 & j=2 & j=3 & j=4 \\ \hline
		i=0 &
		\frac{1}{\sqrt{5}}
		\begin{pmatrix}
			1\\1\\1\\1\\1
		\end{pmatrix}
		&
		\frac{1}{\sqrt{5}}
		\begin{pmatrix}
			1\\ \omega\\ \omega^2\\ \omega^3\\ \omega^4
		\end{pmatrix}
		&
		\frac{1}{\sqrt{5}}
		\begin{pmatrix}
			1\\ \omega^2\\ \omega^4\\ \omega\\ \omega^3
		\end{pmatrix}
		&
		\frac{1}{\sqrt{5}}
		\begin{pmatrix}
			1\\ \omega^3\\ \omega\\ \omega^4\\ \omega^2
		\end{pmatrix}
		&
		\frac{1}{\sqrt{5}}
		\begin{pmatrix}
			1\\ \omega^4\\ \omega^3\\ \omega^2\\ \omega
		\end{pmatrix}
		\\[3ex] \hline
		i=1 &
		\frac{1}{\sqrt{5}}
		\begin{pmatrix}
			1\\ \omega^2\\ \omega\\ \omega^4\\ \omega^3
		\end{pmatrix}
		&
		\frac{1}{\sqrt{5}}
		\begin{pmatrix}
			1\\ \omega^3\\ \omega^3\\ \omega^2\\ \omega^2
		\end{pmatrix}
		&
		\frac{1}{\sqrt{5}}
		\begin{pmatrix}
			1\\ \omega^4\\ 1\\ 1\\ \omega
		\end{pmatrix}
		&
		\frac{1}{\sqrt{5}}
		\begin{pmatrix}
			1\\ 1\\ \omega^2\\ \omega^3\\ 1
		\end{pmatrix}
		&
		\frac{1}{\sqrt{5}}
		\begin{pmatrix}
			1\\ \omega\\ \omega^4\\ \omega\\ \omega^4
		\end{pmatrix}
		\\[3ex] \hline
		i=2 &
		\frac{1}{\sqrt{5}}
		\begin{pmatrix}
			1\\ \omega^4\\ \omega^2\\ \omega^3\\ \omega
		\end{pmatrix}
		&
		\frac{1}{\sqrt{5}}
		\begin{pmatrix}
			1\\ 1\\ \omega^4\\ \omega\\ 1
		\end{pmatrix}
		&
		\frac{1}{\sqrt{5}}
		\begin{pmatrix}
			1\\ \omega\\ \omega\\ \omega^4\\ \omega^4
		\end{pmatrix}
		&
		\frac{1}{\sqrt{5}}
		\begin{pmatrix}
			1\\ \omega^2\\ \omega^3\\ \omega^2\\ \omega^3
		\end{pmatrix}
		&
		\frac{1}{\sqrt{5}}
		\begin{pmatrix}
			1\\ \omega^3\\ 1\\ 1\\ \omega^2
		\end{pmatrix}
		\\[3ex] \hline
		i=3 &
		\frac{1}{\sqrt{5}}
		\begin{pmatrix}
			1\\ \omega\\ \omega^3\\ \omega^2\\ \omega^4
		\end{pmatrix}
		&
		\frac{1}{\sqrt{5}}
		\begin{pmatrix}
			1\\ \omega^2\\ 1\\ 1\\ \omega^3
		\end{pmatrix}
		&
		\frac{1}{\sqrt{5}}
		\begin{pmatrix}
			1\\ \omega^3\\ \omega^2\\ \omega^3\\ \omega^2
		\end{pmatrix}
		&
		\frac{1}{\sqrt{5}}
		\begin{pmatrix}
			1\\ \omega^4\\ \omega^4\\ \omega\\ \omega
		\end{pmatrix}
		&
		\frac{1}{\sqrt{5}}
		\begin{pmatrix}
			1\\ 1\\ \omega\\ \omega^4\\ 1
		\end{pmatrix}
		\\[3ex] \hline
		i=4 &
		\frac{1}{\sqrt{5}}
		\begin{pmatrix}
			1\\ \omega^3\\ \omega^4\\ \omega\\ \omega^2
		\end{pmatrix}
		&
		\frac{1}{\sqrt{5}}
		\begin{pmatrix}
			1\\ \omega^4\\ \omega\\ \omega^4\\ \omega
		\end{pmatrix}
		&
		\frac{1}{\sqrt{5}}
		\begin{pmatrix}
			1\\ 1\\ \omega^3\\ \omega^2\\ 1
		\end{pmatrix}
		&
		\frac{1}{\sqrt{5}}
		\begin{pmatrix}
			1\\ \omega\\ 1\\ 1\\ \omega^4
		\end{pmatrix}
		&
		\frac{1}{\sqrt{5}}
		\begin{pmatrix}
			1\\ \omega^2\\ \omega^2\\ \omega^3\\ \omega^3
		\end{pmatrix}
		\\ \hline
	\end{array}
	\]
	where \(\omega=e^{2\pi\mathrm{i}/5}\).
	Let $U=\frac{1}{\sqrt5}(\omega^{jk})_{j,k\in\mathbb{Z}_5}$ be the  $0$-th row of $L$. Define $\tilde L=(|\tilde\ell_{i,j}\rangle)$ by
	\[
	|\tilde\ell_{i,j}\rangle = U^*|\ell_{i,j}\rangle, \qquad i,j\in\mathbb{Z}_5.
	\]
	Then $\tilde L$ is a $\mathrm{QLS}(5)$ over the real field with maximum cardinality $25$.
	
	Substituting $\omega+\omega^4=\frac{\sqrt{5}-1}{2}$ and $\omega^2+\omega^3=\frac{-\sqrt{5}-1}{2}$ into $|\tilde\ell_{i,j}\rangle$ yields the following real array:
	\[
	\begin{array}{c|c|c|c|c|c|}
		& j=0 & j=1 & j=2 & j=3 & j=4 \\ \hline
		i=0 &
		\begin{pmatrix} 1 \\ 0 \\ 0 \\ 0 \\ 0 \end{pmatrix} &
		\begin{pmatrix} 0 \\ 1 \\ 0 \\ 0 \\ 0 \end{pmatrix} &
		\begin{pmatrix} 0 \\ 0 \\ 1 \\ 0 \\ 0 \end{pmatrix} &
		\begin{pmatrix} 0 \\ 0 \\ 0 \\ 1 \\ 0 \end{pmatrix} &
		\begin{pmatrix} 0 \\ 0 \\ 0 \\ 0 \\ 1 \end{pmatrix} \\[4ex] \hline
		i=1 &
		\frac{\sqrt{5}}{10}\begin{pmatrix} 0 \\ 2 \\ \sqrt{5}-1 \\ \sqrt{5}+1 \\ -2 \end{pmatrix} &
		\frac{\sqrt{5}}{10}\begin{pmatrix} -2 \\ 0 \\ 2 \\ \sqrt{5}-1 \\ \sqrt{5}+1 \end{pmatrix} &
		\frac{\sqrt{5}}{10}\begin{pmatrix} \sqrt{5}+1 \\ -2 \\ 0 \\ 2 \\ \sqrt{5}-1 \end{pmatrix} &
		\frac{\sqrt{5}}{10}\begin{pmatrix} \sqrt{5}-1 \\ \sqrt{5}+1 \\ -2 \\ 0 \\ 2 \end{pmatrix} &
		\frac{\sqrt{5}}{10}\begin{pmatrix} 2 \\ \sqrt{5}-1 \\ \sqrt{5}+1 \\ -2 \\ 0 \end{pmatrix} \\[4ex] \hline
		i=2 &
		\frac{\sqrt{5}}{10}\begin{pmatrix} 0 \\ \sqrt{5}-1 \\ -2 \\ 2 \\ \sqrt{5}+1 \end{pmatrix} &
		\frac{\sqrt{5}}{10}\begin{pmatrix} \sqrt{5}+1 \\ 0 \\ \sqrt{5}-1 \\ -2 \\ 2 \end{pmatrix} &
		\frac{\sqrt{5}}{10}\begin{pmatrix} 2 \\ \sqrt{5}+1 \\ 0 \\ \sqrt{5}-1 \\ -2 \end{pmatrix} &
		\frac{\sqrt{5}}{10}\begin{pmatrix} -2 \\ 2 \\ \sqrt{5}+1 \\ 0 \\ \sqrt{5}-1 \end{pmatrix} &
		\frac{\sqrt{5}}{10}\begin{pmatrix} \sqrt{5}-1 \\ -2 \\ 2 \\ \sqrt{5}+1 \\ 0 \end{pmatrix} \\[4ex] \hline
		i=3 &
		\frac{\sqrt{5}}{10}\begin{pmatrix} 0 \\ \sqrt{5}+1 \\ 2 \\ -2 \\ \sqrt{5}-1 \end{pmatrix} &
		\frac{\sqrt{5}}{10}\begin{pmatrix} \sqrt{5}-1 \\ 0 \\ \sqrt{5}+1 \\ 2 \\ -2 \end{pmatrix} &
		\frac{\sqrt{5}}{10}\begin{pmatrix} -2 \\ \sqrt{5}-1 \\ 0 \\ \sqrt{5}+1 \\ 2 \end{pmatrix} &
		\frac{\sqrt{5}}{10}\begin{pmatrix} 2 \\ -2 \\ \sqrt{5}-1 \\ 0 \\ \sqrt{5}+1 \end{pmatrix} &
		\frac{\sqrt{5}}{10}\begin{pmatrix} \sqrt{5}+1 \\ 2 \\ -2 \\ \sqrt{5}-1 \\ 0 \end{pmatrix} \\[4ex] \hline
		i=4 &
		\frac{\sqrt{5}}{10}\begin{pmatrix} 0 \\ -2 \\ \sqrt{5}+1 \\ \sqrt{5}-1 \\ 2 \end{pmatrix} &
		\frac{\sqrt{5}}{10}\begin{pmatrix} 2 \\ 0 \\ -2 \\ \sqrt{5}+1 \\ \sqrt{5}-1 \end{pmatrix} &
		\frac{\sqrt{5}}{10}\begin{pmatrix} \sqrt{5}-1 \\ 2 \\ 0 \\ -2 \\ \sqrt{5}+1 \end{pmatrix} &
		\frac{\sqrt{5}}{10}\begin{pmatrix} \sqrt{5}+1 \\ \sqrt{5}-1 \\ 2 \\ 0 \\ -2 \end{pmatrix} &
		\frac{\sqrt{5}}{10}\begin{pmatrix} -2 \\ \sqrt{5}+1 \\ \sqrt{5}-1 \\ 2 \\ 0 \end{pmatrix} \\
		\hline
	\end{array}
	\]
\end{example}

\begin{proof}[\bf Proof of Theorem~\ref{thm:1.1}]
	For $q=5$, an explicit real example is given in Example~\ref{ex:q5}. For $q=4$ and $q=6$, see the Appendix (Examples~\ref{ap:q4} and~\ref{ap:q6}, respectively).
	
	For $q\ge 7 $ , let $\sigma=(2\;\;3)(-2\;\;-3)$ be a permutation on $\mathbb{Z}_q$, then $L$ is a $\mathrm{QLS}(q)$ by Theorem~\ref{thm:2.1}. We will prove that $L$ is a $\mathrm{QLS}(q)$ with maximum cardinality and it is equivalent to a quantum Latin square over the real field.
	
	Suppose that $L$ is not a $\mathrm{QLS}(q)$ with maximum cardinality. Then there exist $(i,j)\neq(i',j')$ and $\theta\in[0,2\pi)$ such that
	$
	|\ell_{i,j}\rangle=e^{\mathrm{i}\theta}|\ell_{i',j'}\rangle.
	$
	Take $k=0$, we have $|\ell_{i,j}\rangle^{(0)}=e^{\mathrm{i}\theta}|\ell_{i',j'}\rangle^{(0)}$, so $\frac{1}{\sqrt{q}}=e^{\mathrm{i}\theta}\frac{1}{\sqrt{q}}$, that means $e^{\mathrm{i}\theta}=1$. Hence, for every $k\in\mathbb{Z}_q\setminus \{0\}$, we have $|\ell_{i,j}\rangle^{(k)}=|\ell_{i',j'}\rangle^{(k)}$, i.e. ,
	$
	(i-i')\sigma(k)+(j-j')k\equiv0\pmod q.
	$
	Since $\sigma=(2\;\;3)(-2\;\;-3)$, we have $\sigma(1)=1$, $\sigma(2)=3$, $\sigma(3)=2$. Taking $k=1,2$, we obtain
	$
	x+y\equiv0\pmod q,
	3x+2y\equiv0\pmod q.
	$
	From $x+y\equiv0$, we get $y\equiv -x\pmod q$. Substituting into $3x+2y\equiv0$ gives $x\equiv0\pmod q$, and hence $y\equiv0\pmod q$. Thus $(i,j)=(i',j')$, a contradiction. Therefore $L$ has cardinality $q^2$.
	
	It remains to show that $L$ is equivalent to a real quantum Latin square. Let $U=\frac{1}{\sqrt q}(\omega^{jk})_{j,k\in\mathbb{Z}_q}$ be the $0$-th row of $L$. Define a new array $\tilde L$ by
	$
	|\tilde\ell_{i,j}\rangle=U^*|\ell_{i,j}\rangle, i,j\in\mathbb{Z}_q.
$
	Since $U^*$ is unitary, $\tilde L$ is also a $\mathrm{QLS}(q)$ with cardinality $q^2$.
	For any $i,j,r\in\mathbb{Z}_q$, we compute
	\[
	|\tilde\ell_{i,j}\rangle^{(r)}
	=(U^*|\ell_{i,j}\rangle)^{(r)}
	=\frac{1}{q}\sum_{k=0}^{q-1}\omega^{i\sigma(k)+(j-r)k}.
	\]
	Using the fact that for any $k\in \mathbb{Z}_q$, $\sigma(-k)=-\sigma(k)$,
	\[
	\overline{|\tilde\ell_{i,j}\rangle^{(r)}}
	=\frac{1}{q}\sum_{k=0}^{q-1}\omega^{-i\sigma(k)-(j-r)k}
	=\frac{1}{q}\sum_{k=0}^{q-1}\omega^{i\sigma(-k)+(j-r)(-k)}
	=|\tilde\ell_{i,j}\rangle^{(r)}.
	\]
	Thus $|\tilde\ell_{i,j}\rangle_r\in\mathbb R$ for all $i,j,r$. Therefore $\tilde L$ is a $\mathrm{QLS}(q)$ over the real field with maximum cardinality $q^2$. This completes the proof.
\end{proof}

\section{Quantum Latin Hypercube With Maximal Cardinality}

Before proving Theorem~\ref{thm:1.2}, we first establish a lemma that will be used in its proof.

\begin{lemma}
	\label{lem:main}
	For any positive integers $n$ and $q\ge n+2$, there exist permutations $\sigma_0,\sigma_1,\dots,\sigma_{n-1}$ on $\mathbb{Z}_q$ such that the system of equations
	\[
	x_0\sigma_0(k)+x_1\sigma_1(k)+\cdots+x_{n-1}\sigma_{n-1}(k)\equiv0\pmod{q},\qquad k\in\mathbb{Z}_q
	\]
	has only the trivial solution, i.e., $x_0\equiv x_1\equiv\cdots\equiv x_{n-1}\equiv0\pmod{q}$.
\end{lemma}
\begin{proof}
	Let $\sigma_i=(i+2\;\;i+3)$ for $i=0,1,\dots,n-2$, and let $\sigma_{n-1}$ be the identity permutation on $\mathbb{Z}_q$. Taking $k=1,2,\dots,n$ in turn, we have:
\begin{align*}
	x_0+x_1+\cdots+x_{n-1}&\equiv0\pmod{q},\tag{E1} \\
	3x_0+2x_1+2x_2+\cdots+2x_{n-1}&\equiv0\pmod{q},\tag{E2}\\
	2x_0+4x_1+3x_2+3x_3+\cdots+3x_{n-1}&\equiv0\pmod{q},\tag{E3}\\
	&\;\;\vdots \\
	k\sum_{i=0}^{k-4}x_i+(k-1)x_{k-3}+(k+1)x_{k-2}+k\sum_{i=k-1}^{n-1}x_i&\equiv0\pmod{q},\tag{Ek}\\
	&\;\;\vdots \\
	n x_0+n x_1+\cdots+n x_{n-4}+(n-1)x_{n-3}+(n+1)x_{n-2}+n x_{n-1}&\equiv0\pmod{q}.\tag{En}
\end{align*}
Subtracting $2$ times (E1) from (E2) gives $x_0 \equiv 0 \mod{q}$.
	For $3\le k \le n$, subtracting $k$ times (E1) from (Ek) gives $x_{k-2} \equiv x_{k-3} \pmod{q}$.
	So we have $x_0\equiv x_1\equiv\cdots\equiv x_{n-2}\equiv0\pmod{q}$.
	Substituting these into (E1) yields $x_{n-1}\equiv 0 \pmod{q}$.
	Therefore, the system has only the trivial solution. 
\end{proof}

\begin{proof}[\bf Proof of Theorem~\ref{thm:1.2}]
	
	Let $C=(|c_{i_0,i_1,\dots ,i_{n-1}}\rangle)$ be an $n$-dimensional array of size $q\times q\times\dots \times q$, where
	$$|c_{i_0,i_1,\dots ,i_{n-1}}\rangle ^{(k)}=\frac{1}{\sqrt{q}}\omega^{i_0\sigma_0(k)+i_1\sigma_1(k)+\dots +i_{n-1}\sigma_{n-1}(k)},$$
	$
	\omega = e^{2\pi\mathrm{i}/q}$,$\sigma_i=(i+2\;\;i+3)$ for $i=0,1,\dots,n-2$, and $\sigma_{n-1}$ is the identity permutation on $\mathbb{Z}_q$.
	
	Firstly, we prove that $C$ is a $\mathrm{QLH}(q,n)$. For any $i_t, i_t' \in \mathbb{Z}_q,t\in\{0,1,\dots,n-1\}$, 
	\begin{align*}
		&\langle c_{i_0,\dots,i_t,\dots ,i_{n-1}}|c_{i_0,\dots ,i_t' ,\dots ,i_{n-1}}\rangle\\
		=&\frac{1}{q}\sum_{k=0}^{q-1}\overline{\omega^{i_0\sigma_0(k)+\dots +i_t\sigma_t(k)+\dots +i_{n-1}\sigma_{n-1}(k)}}\omega^{i_0\sigma_0(k)+\dots +i_t'\sigma_t(k)+\dots +i_{n-1}\sigma_{n-1}(k)} \\
		=&\frac{1}{q}\sum_{k=0}^{q-1}\omega^{-(i_0\sigma_0(k)+\dots +i_t\sigma_t(k)+\dots +i_{n-1}\sigma_{n-1}(k))}\omega^{i_0\sigma_0(k)+\dots +i_t'\sigma_t(k)+\dots +i_{n-1}\sigma_{n-1}(k)} \\
		=&\frac{1}{q}\sum_{k=0}^{q-1}\omega^{(i_t'-i_t)\sigma_t(k)} \\
		=&\frac{1}{q}\sum_{k=0}^{q-1}\omega^{(i_t'-i_t)k} \\
		=&\delta_{i_t'i_t}.
	\end{align*}
	Thus for any axis direction $t$, the set
	$
	\{|c_{i_0,\dots , i_{t-1},x,i_{t+1},\dots ,i_{n-1}}\rangle|x\in [q]\}
	$
	is an orthonormal basis of $\mathcal{H}_q$ and $C$ is a $\mathrm{QLH}(q,n)$.
	
	Next, we suppose that $C$ is not a $\mathrm{QLH}(q,n)$ with maximal cardinality. Then
	there exist $(i_0,i_1,\dots ,i_{n-1})\not= (i_0',i_1',\dots ,i_{n-1}')$ such that $|c_{i_0,i_1,\dots ,i_{n-1}}\rangle=e^{i\theta }|c_{i_0',i_1',\dots ,i_{n-1}'}\rangle$.
	So, for any $k\in \mathbb{Z}_q$, we have that
	$$
	\omega^{i_0\sigma_0(k)+i_1\sigma_1(k)+\dots +i_{n-1}\sigma_{n-1}(k)}=e^{i\theta}\omega^{i_0'\sigma_0(k)+i_1'\sigma_1(k)+\dots +i_{n-1}'\sigma_{n-1}(k)}.
	$$
	Take $k=0$, we have $e^{i\theta}=1$. Further, for any $k \in \{1,2,\dots,n\} \subset \{1,2,\dots,q-1\}$ ,
	$$
	(i_0-i_0')\sigma_0(k)+(i_1-i_1')\sigma_1(k)+\dots +(i_{n-1}-i_{n-1}')\sigma_{n-1}(k) \equiv 0 \mod{q}.
	$$
	So, $(i_0-i_0',i_1-i_1',\dots , i_{n-1}-i_{n-1}')\not = \mathbf{0}_n$ is one solution of the system of equations, which contradicts the property of $\sigma_0,\sigma_1,\dots \sigma_{n-1}$ by  Lemma~\ref{lem:main}. That means that $C$ is  a $\mathrm{QLH}(q,n)$ with maximal cardinality.
\end{proof}

In particular, when $q=3$, there exists a quantum Latin cube of order
$n$ with maximal cardinality for $n\ge 5$. Combined with the existence of quantum Latin cube of order
$4$ with maximal cardinality in \cite{MR5034584}, we obtain the following corollary:

\begin{corollary}
	For any positive integers $q\ge 4$, there exists a $\mathrm{QLC}(q)$ with maximal cardinality $q^3$.
\end{corollary}

\noindent{\bf Acknowledgments}  H. Cao's research was supported by the National Natural Science Foundation of China (Grants No. 12471313 and No. 12071226).

\appendix

\section{Explicit examples for $q=4$ and $q=6$}

\begin{example}
	\label{ap:q4}
	For $q=4$, the following array is a $\mathrm{QLS}(4)$ over the real field with maximum cardinality $16$:
	\[
	\begin{array}{|c|c|c|c|}
		\hline
		|0\rangle & |1\rangle & \frac{1}{\sqrt{2}}(|2\rangle+|3\rangle) & \frac{1}{\sqrt{2}}(|2\rangle-|3\rangle) \\
		\hline
		|2\rangle &
		|3\rangle &
		\frac{1}{\sqrt{2}}(|0\rangle+|1\rangle) &
		\frac{1}{\sqrt{2}}(|0\rangle-|1\rangle) \\
		\hline
		\frac{1}{\sqrt{2}}(|1\rangle-|3\rangle) &
		\frac{1}{\sqrt{2}}(|0\rangle-|2\rangle) &
		\frac{1}{2}(|0\rangle-|1\rangle+|2\rangle-|3\rangle) &
		\frac{1}{2}(|0\rangle+|1\rangle+|2\rangle+|3\rangle) \\
		\hline
		\frac{1}{\sqrt{2}}(|1\rangle+|3\rangle) &
		\frac{1}{\sqrt{2}}(|0\rangle+|2\rangle) &
		\frac{1}{2}(|0\rangle-|1\rangle-|2\rangle+|3\rangle) &
		\frac{1}{2}(|0\rangle+|1\rangle-|2\rangle-|3\rangle) \\
		\hline
	\end{array}
	\]
	This example was previously observed in \cite{PhysRevA.104.042423}.
\end{example}

\begin{example}
	\label{ap:q6}
	For $q=6$, consider the following two row-quantum Latin rectangles over the reals, each with maximal cardinality:
	\[
	R=
	\begin{bmatrix}
		\frac{1}{\sqrt{2}}(|0\rangle+|1\rangle) & \frac{1}{\sqrt{2}}(|0\rangle-|1\rangle) & |2\rangle \\
		\frac{1}{\sqrt{2}}(|0\rangle+|2\rangle) & \frac{1}{\sqrt{2}}(|0\rangle-|2\rangle) & |1\rangle
	\end{bmatrix},\qquad
	S=
	\begin{bmatrix}
		|0\rangle & |1\rangle \\
		\frac{\sqrt{3}}{2}|0\rangle+\frac{1}{2}|1\rangle & -\frac{1}{2}|0\rangle+\frac{\sqrt{3}}{2}|1\rangle \\
		\frac{1}{2}|0\rangle+\frac{\sqrt{3}}{2}|1\rangle & -\frac{\sqrt{3}}{2}|0\rangle+\frac{1}{2}|1\rangle
	\end{bmatrix}.
	\]
	Denote the entries of $R$ and $S$ by $R=(|r_{i,j}\rangle)_{i\in\{0,1\},\,j\in\{0,1,2\}}$ and $S=(|s_{j,l}\rangle)_{j\in\{0,1,2\},\,l\in\{0,1\}}$, respectively.
	
	Let $W$ be the quantum Latin square of order $6$ obtained by applying Theorem~3.1 of \cite{ZHANG2026114863} to $R$ and $S$.  Its entry in the $k$-th row and $l$-th column of the $(i,j)$-th block is
	\[
	|w_{i,j,k,l}\rangle = |r_{i,j+k\pmod3}\rangle \otimes |s_{j,i+l\pmod2}\rangle,
	\]
	where $i,l\in\{0,1\}$, $j,k\in\{0,1,2\}$.
	
	Since $R$ and $S$ consist of real vectors, $W$ gives a quantum Latin square of order $6$ over the real field with maximal cardinality~$36$.  Explicit enumeration of all $36$ vectors is omitted for brevity.
\end{example}

\end{document}